\documentclass[12pt,a4paper]{amsart}

\usepackage[parfill]{parskip}

\theoremstyle{definition}

\newtheorem{definition}{Definition}
\theoremstyle{remark}

\theoremstyle{plain}
\newtheorem{theorem}{Theorem}
\newtheorem{lemma}[theorem]{Lemma}

\newtheorem{cor}[theorem]{Corollary}

\theoremstyle{definition}

\newcommand\xqed[1]{%
    \leavevmode\unskip\penalty9999 \hbox{}\nobreak\hfill
    \quad\hbox{#1}}

\newtheorem{exampleqqq}[theorem]{Example}

\newtheorem{remarkqqq}[definition]{Remark}
\newenvironment{remark}{\begin{remarkqqq}}
 {\xqed{$\triangle$}\end{remarkqqq}}

\usepackage[utf8]{inputenc}
\usepackage{hyperref}
\usepackage{amsfonts,amssymb,amsmath}
\usepackage[mathscr]{euscript}
\usepackage{enumerate,xspace}
\usepackage{graphicx}
\usepackage{verbatim}

\renewcommand{\d}{\,d}

\providecommand{\RR}{\mathbb{R}}

\providecommand{\ZZ}{\mathbb{Z}}

\providecommand{\NN}{\mathbb{N}}
\providecommand{\CF}{\mathscr{F}}
\providecommand{\CB}{\mathscr{B}}

\providecommand{\CM}{\mathscr{M}}

\providecommand{\CG}{\mathscr{G}}

\providecommand{\opn}{\operatorname}

\def\ii#1{^{(#1)}}
\def\pp#1{\left(#1\right)}

\renewcommand{\P}{\mathsf{P}}

\providecommand{\E}{\mathsf{E}}
\providecommand{\Ex}[1]{\E\pp{#1}}

\providecommand{\var}{\opn{var}}

\providecommand{\Log}{\opn{Log}}

\providecommand{\ett}{\mathsf{1}}

\def\tl#1{\tilde#1}
\def\percol{P_\infty}
\def\g{\mathit{G}}
\def\e{\mathit{H}}
\def\ww{W}
\def\tle{\tilde\e}

\providecommand{\w}{\omega} 

\renewcommand{\L}{\mathscr{L}}

\def\cork{\operatorname{corank}}

\def\Firr{{F_{R}}}
\def\halfopen#1#2{\ensuremath{\left[#1,#2\right)}}
\def\Fn{{\halfopen0n}}
\def\id{\operatorname{id}}
\def\X{X}
\def\Y{Y}
\def\T{T}
\def\L{\mathscr{L}}
\def\A{\mathscr{A}} 
\def\cyl#1#2{{[#1]}_{#2}}

\def\J{\mathrm{J}}

\def\bsmash#1{\vphantom{\bar{#1}}{#1}}
\def\C{\mathscr{C}} 
\def\RC{\mathsf{RC}} 
\def\FK{\mathsf{FK}}
\def\orig{\mathrm{o}}
\def\cyl#1#2{[#1]_{#2}}

\author{Anders Johansson, Anders \"Oberg, and Mark Pollicott}
\address{Anders Johansson, Department of Mathematics, University of G\"avle, 801 76
G\"avle, Sweden.} \email{ajj@hig.se}
\address{Anders \"Oberg, Department of Mathematics, Uppsala University, P.O.\ Box 480,
751 06 Uppsala, Sweden.} \email{anders@math.uu.se}
\address{Mark Pollicott, Mathematics Institute, University of Warwick, Coventry, CV4 7AL,
UK.} \email{mpollic@maths.warwick.ac.uk}
\date{}

\begin{document}
\title{Continuous eigenfunctions of the transfer operator for Dyson models}

\begin{abstract}\noindent
  In this article we address a well known problem at the intersection of ergodic theory and statistical mechanics.
  We prove that there exists a continuous eigenfunction for the transfer
  operator corresponding to pair potentials that satisfy a square summability
  condition on the variations, when the inverse temperature is subcritical. As a corollary we 
  obtain a continuous eigenfunction for the classical Dyson model, with 
  interactions $\J(k)=\beta \, k^{-\alpha}$, $k\ge1$, in the whole subcritical 
  regime $\beta<\beta_c$ for which the parameter $\alpha$ is greater than $3/2$.
\end{abstract}

\maketitle

\section{Introduction}\noindent
Ruelle~\cite{ruelle} and Sinai~\cite{sin} pioneered the study of long-range
models within statistical mechanics in terms of transfer operators and their
eigenfunctions and eigenmeasures, an important theory that later was further developed by
Walters~\cite{walters1}.

It is well known~\cite{walters1} that there exists a continuous and strictly
positive eigenfunction of a transfer operator defined on a symbolic shift space
with a finite number of symbols if the one-point potential $\phi$ has summable
variations. Here we prove, for the particular class of potentials we study, the existence of a continuous
eigenfunction under the weaker condition of square summable variations of the potential $\phi$ and a subcritical scaling parameter $\beta$.

\def\Spin{\mathcal{S}}
Let $\Spin =\{-1,+1\}$ and $V=\NN=\{0,1,\dots\}$ or $V=\ZZ$. Elements $x\in\X$ of $X=\Spin^\NN$ 
are \emph{one-sided} \emph{spin sequences}. 
We will also consider \emph{two-sided} spin sequences $x\in\bar\X=\Spin^\ZZ$ 
depending on the context. 
Let $\T:\X\to\X$ be the left shift $(x_i) \mapsto (x_{i+1})$. 
Consider the positive operator $\L=\L_{\phi}$ on the space of continuous function
$C(\X)$, given by
\[
	\L f(x)= \sum_{y\in \T^{-1}x} e^{\phi(y)}\, f(y),
\]
In this paper, we specify the
\emph{one-point potential} $\phi(x)=\phi(x;\J)\in C(\X)$ by a sequence $\J(k)\ge0$,
$\J(0)=0$, so that $\phi$ takes the form
\begin{equation}\label{dyson1pt}
	\phi(x) = x_0\cdot \sum_{k=1}^\infty \J(k) x_k.
\end{equation}
This type of potential appears naturally in the context of statistical mechanics;
in the important special class of \emph{Dyson potentials} 
we have $\J(k)=\beta \cdot k^{-\alpha}$ for the interaction strength power $\alpha>1$ 
and the inverse temperature $\beta>0$.
We assume the sequence $\J(k)$ is summable, so that in particular, for $n\ge0$,
\[
  r_n := \sum_{k=n+1}^\infty \J(k) \to 0 \quad\text{as $n\to\infty$}.
\]
Note that the $n$th variation $\var_n \phi$, as defined in section \ref{sec:prelim}, equals $2 r_n$. 
Let $\CM(\X)$ denote the set of probability measures on $\X$ and let
\[
  \CM_\phi = \{\nu\in\CM(\X) : \forall\, f\in C(\X) \, \int \L_\phi f\d\nu = \lambda \int f\d\nu, \lambda > 0\}
\]
be the set of normalised \emph{eigenmeasures} for the unique positive eigenvalue
$\lambda=\lambda_\phi$ that equals the spectral radius of $\L$; we also write $\L^*\nu=\lambda \nu$. 
From the Schauder-Tychonoff theorem applied to the mapping 
$\nu\mapsto\L^*\nu/\nu(\X)$ on $\CM(\X)$, we deduce that $\CM_\phi\not=\emptyset$.

We prove the following general result.
\begin{theorem}\label{thm:main}
  For $x\in\X$, define $r(x) = \sum_{n=0}^\infty r_n x_n$. If $\nu\in\CM_\phi$ and
  \begin{equation}\label{eq:coshx}
	\int e^{r(x)} \d\nu(x) < \infty
  \end{equation}
  then there is a strictly positive \emph{continuous} eigenfunction
  $h(x)\in C(\X)$ of $\L=\L_\phi$ such that $\L h = \lambda h$.
\end{theorem}

A measure $\mu\in\CM(\X)$ is a Doeblin measure (\cite{berger2}, a.k.a\ as a
$g$-measure \cite{keane}) if it is a \emph{translation invariant} eigenmeasure of
$\L_{\log g}^*$ for some continuous function $g>0$ with $\L_{\log g} 1 = 1$. The
theory of Doeblin measures (\cite{doeblin, keane, johob, berger, johob2}) is
close to the topic of this paper, since we can construct a Doeblin function $g$,
from a one-point potential $\phi$ and a continuous eigenfunction $h(x)$ of the
transfer operator $\L_\phi$ by
\begin{equation}\label{DF}
	g(x)=\frac{e^{\phi(x)}\, h(x)}{\lambda\, h(Tx)}.
\end{equation}
From~\eqref{DF}, we see that the measure $\mu(x) = h(x)\cdot \nu$ is a translation
invariant eigenmeasure to the transfer operator $\L_{\log g}$, i.e., a Doeblin
measure. We refer to $\mu$ as the \emph{equilibrium} measure of $\phi$. Thus the
existence of a continuous eigenfunction of the transfer operator implies the
\emph{$g$-measure property} of $\mu$. i.e., that $\mu$ can be represented as a Doeblin measure 
(a $g$-measure) for some continuous function $g>0$, i.e., so that it is a fixed point of $\L_{\log g}^*$. In contrast to our result, we
observe the result by Bissacot \emph{et al.}~\cite{vanenter}, where they show
that the $g$-measure property does not hold in the context of the Dyson model, when $1<\alpha<2$ is small enough and for
high values of $\beta$. In view of \eqref{DF} this means that there is no 
continuous eigenfunction. Note, however, that the $g$-measure property may hold even if a
continuous eigenfunction does not exist. 

We may also construct $\nu$ and $\mu$ as long-range \emph{Ising models}. For $V=\ZZ$
and $V=\NN$ define the \emph{potential} (see section~\ref{sec:potentials} below) $\Phi_V$ 
\[
  \Phi_V(x) = \sum_{ij} J_V(ij) x_i x_j, \quad x\in \Spin^V
\]
with $J_V(ij)=\J(|i-j|)$ and where we sum over the set $ij\in V\ii2$ of unordered
pairs of elements $i,j\in V$. We have
\[
  \Phi_V(x) = \sum_{k\in V} \phi(T^k x)
\]
as an equality where we consider both the left and right hand side as potentials. 

The set $\CM_{\phi}$ of eigenmeasures is equal to the
set of Gibbs measures $\CG(\Phi_\NN)$ (one-sided Ising models) consistent with the
potential $\Phi_\NN(x)$. A two-sided Ising model $\mu(\bar x)$,
$\bar x\in{\{-1,+1\}}^\ZZ$, is a translation invariant Gibbs measure in
$\mu\in\CG(\Phi_\ZZ)$. By the Fortuin-Kasteleyn (FK) correspondence, we can couple
$\nu(x)$ with a random cluster distribution $\nu(\g)$ on graphs $\g$ with vertex
set $\NN$ and, similarly, couple $\mu(\bar x)$ to a random cluster distribution
$\mu(\g)$ on graphs with vertex set $\ZZ$. We write $\mu(x)$ for the equilibrium measure, which we can
capture as the marginal distribution of $x=\bar x\vert_\NN$
under $\mu$. For fixed interactions $J_V(ij)$, it is well-known that there
is a critical $\beta_c(J_\NN)\in [0,\infty]$ such that $|\CM_{\beta\phi}|=|\CG(\Phi_\NN)| = 1$ if
$0\le \beta<\beta_c(J_\NN)$ (i.e.\ \emph{uniqueness}) and non-uniqueness if $\beta > \beta_c(J_\NN)$.
We also have a critical $\beta_c(J_\ZZ)$ for uniqueness of $\CG(\beta \Phi_\ZZ)$. 
These critical $\beta_\NN$ and $\beta_\ZZ$ are also critical values for the existence of an infinite
cluster (i.e.,~percolation) in the corresponding random cluster models.

To derive the following theorem we use a result of
Hutchcroft~\cite[Theorem~1.5]{Hutch} about exponentially small tail
probabilities in the cluster size distribution of $\mu(\g)$ in the subcritical
regime. It is a generalisation of a result by Duminil-Copin \emph{et
  al.}~\cite{duminil2} to the long-range setting that we are considering. (See
also the paper by Aoun~\cite{aoun}.)
\begin{theorem}\label{main}
  If $\phi$ has square summable variations, i.e.\ if
  \(\sum_{n=0}^\infty {(\var_n \phi )}^2<\infty\) then there exists a unique continuous
  eigenfunction $h$ of $\L_{\beta\phi}$ for all $\beta<\beta_c(J_\ZZ)$.
\end{theorem}
\begin{remark}
  We conjecture that $\beta_c(J_\ZZ)=\beta_c(J_\NN)$. 
\end{remark}
\begin{remark}
The condition of \emph{square summable variations} was studied in the context of uniqueness and non-uniqueness of Doeblin measures ($g$-measures) $\mu$ 
in \cite{johob}, \cite{johob2}, \cite{berger} and \cite{tak}, i.e., when $\phi=\log g$ and $\sum_{y\in T^{-1}x}g(y)=1$ for all $x$ and $\L_{\log g}^*\mu=\mu$. In \cite{johob} and \cite{johob2}
uniqueness of $\mu$ was proved when $\sum_n (\var_n \log g)^2<\infty$, and Berger \emph{et al}.\ \cite{berger} proved that this is sharp in the sense that for all $\varepsilon>0$ we can have $\sum_n (\var_n \log g)^{2+\varepsilon}<\infty$
and multiple solutions $\mu$ of $\L_{\log g}^*\mu=\mu$. In Gallesco \emph{et al}.\ \cite{tak} they study the ramifications of square summability even further. 
\end{remark}

\begin{remark}
In \cite[Theorem 1]{johob4} we proved that  $\beta_c(J_\ZZ)\le 8\beta_c(J_\NN)$,
which, in the light of the counterexample to uniqueness by Dyson \cite{dyson}, shows that for general potentials $\phi$ we would have examples such that for $\varepsilon>0$, $\sum_n (\var_n \phi)^{1+\varepsilon}<\infty$ and with multiple eigenmeasures $\nu$, i.e., multiple solutions 
$\nu$ of $\L^* \nu=\lambda \nu$. Hence it is clear that the condition of summable variations of general potentials is in the above sense sharp for uniqueness of an eigenmeasure $\nu$ of $\L^*$, and similarly the condition of square summable variations is sharp for Doeblin measures ($g$-measures). In view of the conjugation \eqref{DF},
where we under sufficiently strong conditions have a unique Doeblin measure $\mu$ for the potential $\log g$, and a unique eigenmeaure $\nu$ of $\L^*$ for the potential $\phi$, it seems reasonable to think that for general continuous potentials (not only the potentials we consider here), the condition of square summable variations could be a very important condition for the existence of a continuous eigenfunction of the transfer operator, but we do not specify any conjecture in this direction. 
\end{remark}
\begin{remark}
Walters considered \emph{Bowen's condition} and managed to obtain some regularity for an eigenfunction of the transfer operator in \cite[Theorem 5.1]{walters2}. In the case of the Dyson potentials it reduces to summable variations, and it is not clear whether a continuous eigenfunction follows in general from Bowen's condition. Another interesting condition is the one provided by Berbee in \cite{berbee89}. Berbee proves that there exists a unique Gibbs measure for both one-sided and two-sided long-range models whenever
$$\sum_{n=1}^\infty e^{-r_1-\cdots -r_n}=\infty,$$
where $r_n=\var_n \phi$, for some potential $\phi$, suggesting the possible existence of a continuous eigenfunction, but this is unknown. In the case of the Dyson model, Berbee's condition reduces to the assumption $\alpha \geq 2$ (and high temperatures in the case of $\alpha=2$). Note that $\alpha>2$ is the region of summable variations in which
the Dyson model does not have a phase transition at any temperature. 
\end{remark}

\begin{remark}
The existence of a continuous eigenfunction is also related to the rigidity of coboundaries; see especially the work of Quas \cite{quas}. 
\end{remark}

As a corollary, we obtain the existence of a continuous eigenfunction in the
important special class of Dyson potentials where $\J(k)=k^{-\alpha}$ in the
subcritical regime when $\alpha>3/2$. In particular, the potential does not satisfy
the stronger condition of summable variations.
\begin{cor}\label{cor}
  For the Dyson model, where $\J(k)=\J^\alpha(k)= k^{-\alpha}$, $\alpha >1$, we have a continuous eigenfunction $h$ of $\L_{\beta\phi}$ whenever $\alpha>3/2$ and $\beta<\beta_c(J_\ZZ^\alpha)$.
\end{cor}

\begin{remark}
 In the recent paper by van Enter, Fern\'andez, Makhmudov and Verbitskiy~\cite{verb-van-enter2}, inspired by a previous version of the present paper, the authors 
 work in a more general setting where the random cluster representation does not exist. In Corollary \ref{cor}, we obtain a continuous eigenfunction when 
 $\alpha>3/2$ for the full uniqueness region, with respect to the critical inverse temperature, whereas van Enter \emph{et al}.\ \cite{verb-van-enter2} assume in 
 addition the Dobrushin uniqueness condition. 
\end{remark} 

\begin{remark}
We expect that the square summability condition in Theorem~\ref{main} is sharp.
Applied to the Dyson model in Corollary~\ref{cor}, this means that $\alpha>3/2$ would
be sharp for the existence of a continuous eigenfunction. Note the recent
results by Endo, van Enter and Le Ny~\cite{vanny} in this context, as well as
the earlier short version~\cite{vanny0} by the same authors, where $3/2$ first
appears in a related context.
\end{remark} 

\section{The proofs of Theorem~\ref{thm:main} and Theorem~\ref{main}}

\subsection{Preliminaries}\label{sec:prelim}

Let $\A$ be a finite set and $V$ a countable set. The relation $F \Subset V$ states
that $F$ is a finite subset of $V$. We write $\bar F = V\setminus F$ for the complement
of subsets of $V$. A \emph{configuration} is an element $x={(x_i)}_{i\in V}$ of the
product space $\X=\A^V$. If $V=\ZZ$ or $V=\NN$, we let $\T$ denote the left
shift on the symbolic space $\X$, i.e.\ ${(\T x)}_i = x_{i+1}$ with destruction
of $x_0$ if $V=\NN$. We give the space $\X$ the usual product topology and the
associated Borel sigma-algebra $\CF$. For $G \subset V$ and $x\in\X$, we write $x_G\in\A^G$
for the restriction $x\vert_G$ of $x$ to $G$ and $\CF_G$ for the sigma-algebra
generated by $x_G$ and $\mathrm{m}\CF_G$ for the 
set of $\CF_G$-measurable functions. We denote by $\cyl xG$ the cylinder set
$\cyl xG=\{y \mid y_G=x_G\}$.

For a function $f:\X\to\RR$ the \emph{variation} at $\Lambda\subset V$ and $x\in\X$ is
$$\var_\Lambda f(x) = \sup_{y,z\in \cyl x \Lambda} |f(y)-f(z)|$$ and
$\var_\Lambda f = \sup_x \var_\Lambda f(x)$. A function $f$ is \emph{local} if it is
$\CF_\Lambda$-measurable ($\var_\Lambda f = 0$) at some $\Lambda\Subset V$. It is \emph{continuous} at
$x$ if $\lim_{\Lambda_n \uparrow V} \var_{\Lambda_n} f(x) = 0$, where $\Lambda_n \uparrow V$ denotes an
increasing sequence of finite sets whose union is $V$. We denote by $C(\X)$ the
space of continuous functions. In case we consider the integer interval $\Lambda=\Fn$,
we replace subscript by $\Fn$ with the subscript $n$, thus
\[ \var_n f = \var_\Fn f,\quad\text{and}\quad \cyl xn = \cyl x{\Fn}, \ldots \ \text{etc}.\]

Let $\CM(\X)$ denote the space of probability distributions on $\X$. Elements
$\alpha\in\CM(\X)$ are often written as $\alpha(x)$ in order to make it clear that $\alpha$ is
the distribution $\P(x)$ of the random configuration $x\in\X$. If $y=f(x)$ then
$\alpha(y)$ refers to $f_*\alpha=\alpha\circ f^{-1}$. Sometimes, we introduce an underlying
probability space $(\Omega,\CB,\P)$ with expectation operator $\E$, where it is clear
that $\alpha$ is the distribution of $x$. In that case, we write $\E(f(x)) = \int f\d\alpha$
and $\P(x\in A)=\int \ett_A(x) \d\alpha(x)$. We write $\mu(x)\prec\nu(x')$ to mean
\emph{stochastic domination} between elements in $\CM(\X)$, meaning that we can
couple $\mu(x)$ and $\nu(x')$ so that $\P(x \le x')=1$ with respect to the partial
order $\le$ on $\X$ induced by the order on $\A\subset\ZZ$.

A \emph{Bernoulli measure} $\eta(x;p)\in\CM(\X)$ has as parameter an assignment
$p={(p_i)}_{i\in V} \in \CM{(\A)}^V$ and $\eta(x;p)$ is the product measure
$\bigotimes_{i\in V} p_i(x_i)$. We use $\upsilon(x)$ to denote the \emph{uniform
  measure}, i.e.\ $\upsilon(x) = \eta(x;p)$ where $p_i = \upsilon(x_i)$ is the uniform
distribution on $\A$.

\subsubsection{Potentials}\label{sec:potentials}
The Hamming distance on $\X=\A^V$ is the cardinality of
$$\Delta(x,y):=\{i: x(i)\not=y(i)\} $$
and the Hamming graph (digraph) is the set of pairs
$\{(x,x'): |\Delta(x,x')|=1\}$. A \emph{potential} is a real valued function
$\Delta\Phi(x,x')$ on the Hamming graph such that
$$\Delta\Phi(x_1,x_2)+\Delta\Phi(x_2,x_3) +\cdots+ \Delta\Phi(x_k,x_1)=0$$ for
every cycle $\{(x_1,x_2), \dots, (x_k,x_1)\}$. It is clear that potentials on
$\X$ make up a real linear space. Formally, we write potentials as functions
$\Phi(x)$ on $\X$, since any real-valued function $\phi(x)$ defines the
potential $\Delta\phi(x,y)=\phi(x)-\phi(y)$ and we can always represent a
potential \emph{locally} at $\Lambda\Subset V$ with a real valued function
$\Phi_{\Lambda}$ so that $\Delta\Phi(x,y)=\Phi_\Lambda(x)-\Phi_\Lambda(y)$ for
all pairs $(x,y)$ where $\Delta(x,y)\subset\Lambda$. We usually obtain a
potential $\Phi$ as a \emph{potential limit} $\lim \Phi_{\Lambda_n}$ from a
system of local functions $\Phi_\Lambda \in \mathrm{m}\CF_\Lambda$, where the
limits of differences
\[
    \Delta\Phi(x,y) := \lim_{\Lambda_n\uparrow V} \Phi_{\Lambda_n}(x)-\Phi_{\Lambda_n}(y)
\]
are finite and well defined for all pairs $(x,y)$ with $\Delta(x,y)\Subset V$. A
potential $\Phi$ is \emph{continuous} at $x$ if the difference
$\Delta\Phi(x_{\bsmash{\Lambda}}x_{\bar\Lambda},y_{\bsmash{\Lambda}}x_{\bar\Lambda})$
is a continuous function of $x_{\bar\Lambda}$ for fixed $x_\Lambda$ and $y_\Lambda$ in $\A^\Lambda$.

We say that a probability measure $\alpha(x)\in\CM(\X)$ is \emph{fully
  specified} if the system $\{\log \alpha(\cyl x{\Lambda_n})\}$ defines a
potential $\Log\alpha$ so that, for all pairs $x,y$ with $|\Delta(x,y)|=1$, the limit
\[
  \lim_{\Lambda_n\uparrow V}
  \log \alpha(\cyl x{\Lambda_n}) - \log \alpha(\cyl y{\Lambda_n})
\]
exists as a real number. For instance, a Bernoulli measure $\eta(p) = \eta(x;p)$
is fully specified with $\Log \eta(x)$ given by the potential limit
$\sum_{i\in V} \log p_i(x_i)$. We say that $\alpha$ is \emph{consistent} with
potential $\Phi$ if $\Phi=\Log\alpha$ and we write $\CG(\Phi)$ to denote the set
of measures consistent with $\Phi$.

It is well known
that the set $\CG(\Phi)$ of probability measures consistent with potential $\Phi$ is
non-empty whenever $\Phi$ is continuous and the elements of $\CG(\Phi)$ are then said
to be \emph{Gibbsian}. We have \emph{uniqueness} of measures consistent with $\Phi$
if $\CG(\Phi)$ contains only one element.

Given a fully specified measure $\alpha\in\CM(\X)$ and a potential $\Phi$, we write
$e^\Phi\ltimes \alpha$ for the fully specified measures in $\CG(\Phi+\operatorname{Log}\alpha)$.
If we have a unique element in
$e^\Phi\ltimes\alpha$ we write $\mu=e^\Phi\ltimes\alpha$.
If we can represent the
potential $\Phi$ as a function such
that $e^{\Phi}\in L^1(\alpha)$ then $e^{\Phi}\ltimes\alpha$ is absolutely continuous with respect to $\alpha$ and
\begin{equation} \label{ccdef}
    e^{\Phi} \ltimes \alpha = \frac{e^{\Phi} \cdot \alpha}{\int e^{\Phi} \d\alpha}.
\end{equation}
Note that
\begin{equation}\label{eq:modassoc}
  e^\Psi\ltimes (e^\Phi \ltimes \alpha) = e^{\Psi+\Phi}\ltimes\alpha.
\end{equation}
Furthermore, given a product of fully specified measures
$\alpha(x)\otimes \beta(y)$ and potentials $\Phi(x)$ and $\Psi(y)$ on $\X$ and
$\Y$, respectively, we obtain that
\begin{equation}\label{eq:moddistrib}
  e^{\Phi+\Psi}\ltimes (\alpha\otimes\beta) = (e^\Phi\ltimes \alpha)\otimes(e^{\Psi}\ltimes \beta).
\end{equation}

\subsubsection{Graphs}
Let $V\ii2 = V^2/\sim$ denote the set of \emph{unordered pairs}, i.e.\ the
family of equivalence classes for the relation $(i,j)\sim(j,i)$ on $V^2$. For a
map $\varphi:V\to V'$ we write $\varphi\ii2$ for the induced map
$V\ii2\to {V'}\ii2$.
We represent a \emph{graph} (an undirected graph) $G$ on
vertex set $V=V(G)$ as a map $G:E\to V\ii2$ that associates edges in $E=E(G)$ to
pairs of vertices in $V\ii2$.
The graph is \emph{simple} if its representative
map is injective and we call $|G^{-1}(G(e))|$ the multiplicity of the edge
$e\in E(G)$. A graph homomorphism $\varphi:G\to H$ is a pair of maps
$\varphi_E: E(G)\to E(H)$ and $\varphi_V: V(G)\to V(H)$ with commutation rules
$\varphi_V\ii2 \circ G = H \circ \varphi_E$. A map $\varphi:V\to V'$ induces a
\emph{vertex-map} homomorphism $\varphi:G\to G'=\varphi\ii2 \circ G$ given by
the pair $(\id_E,\varphi\ii2)$.

The \emph{complete graph} on $V$, $K(V)$, is the inclusion of the non-loops in
$V\ii2$. Given a bipartition $V=V_-\uplus V_+$ of $V$, the \emph{complete bipartite}
graph $K(V_-,V_+)$ is the inclusion of $V_- \times V_+$ in $V\ii2$. A \emph{path} of
length $n$ in $G$ is an injective graph homomorphism $P_n \to G$ of the graph
$P_n: \{(i,i+1):i\in\Fn\}\hookrightarrow{[0,n]\ii2}$.

A \emph{spanning} subgraph $H$ of $G$ is a restriction of $G$ to a subset
$E(H) \subset E(G)$. Denote by $\Gamma(G)$ the space of spanning subgraphs of $G$ and let
$\Gamma(V)=\Gamma(K(V))$ and $\Gamma(V_-,V_+)=\Gamma(K(V_-,V_+))$. We can represent an element
$\g\in\Gamma(G)$ as a configuration $\g=(\g_{e})\in{\{0,1\}}^{E(G)}$ or, equivalently, as
a subset $\g \subset V\ii2$. Write $G[F] = G\vert_{G^{-1}(F\ii2)}$ for the subgraph
\emph{induced} on vertex set $F \subset V$. All (random) graphs $\g\in\Gamma(V)$ we consider
will (almost surely) have finite degrees, i.e.\
\(\deg(F,\g) := \sum_{i\in F} \sum_{j\in V} \g_{ij} < \infty\), for all $F \Subset V$.

Consider an equivalence relation ${\sim}$ on $V$, where $\pi_{\sim}: i\to i\bmod{\sim}$
denotes the projection onto the equivalence classes. The \emph{contraction}
$G\to G \bmod{\sim}$ of $G$ along ${\sim}$ is the graph homomorphism induced by the
vertex-map $\pi_{\sim}$. Then $G\bmod{\sim}$ has the partition $V/\sim$ of $V$ into
equivalence classes as the vertex set and the same set of edges. Thus the multiplicity of edges may increase. Since $E(G)=E(G\bmod{\sim})$ we have
$\Gamma(G)\cong\Gamma(G\bmod{\sim})$ as configuration spaces. If $F\subset V$ then we write $G^F$ for the
contraction obtained from the equivalence relation ``$x,y \in F$ or $x=y$'', i.e.\
by contracting all vertices in $F$. 

The equivalence relation $i\sim_\g j$ means that there is a path in $\g$ with
endpoints $i,j$. We refer to the equivalence classes $\C(\g):=V(\g)/{\sim_\g}$ as
\emph{clusters} of $\g$. Let $\w(\g)=|\C(\g)|$ be the \emph{number of clusters}.
For infinite graphs and $\Lambda_n\uparrow V$, we define $\w(\g)$ as the potential given by
the (``free boundary'') potential limit of $\w(\g[\Lambda_n])$. We defined the (wired
boundary) potential $\w^w(\g)$ from the limit of $\w(\g^{\bar\Lambda_n})$, where
$\g^{\bar\Lambda_n}$ is the graph $\g$ where all vertices outside $\Lambda_n$ count as one.
The event of \emph{percolation} $\g\in\percol$ means that $\C(\g)$ contains a
cluster of infinite size. The potential $\w(\g)$ is continuous at $\g\in\Gamma(V)$,
precisely when $\g$ contains at most one cluster of infinite size.

We refer to the \emph{rank} of a graph $\g\in\Gamma(V)$ as
$\opn{rank}\g := |V(\g)|-\w(\g)$ and the corank is
$\cork\g = |E(\g)| - \opn{rank} \g$. Then $\cork\g$ is the maximum number of
edges that one may remove from $\g$ without increasing the number of components.
For infinite graphs, we use the induced graphs $\g[\Lambda_n]$ to define the rank and
corank as potentials on $\Gamma(V)$ as potential limits.

\subsection{The eigenfunction as a Radon-Nikodym derivative}

A continuous eigenfunction means that there is a continuous Radon-Nikodym
derivative between the two-sided equilibrium measure (a translation invariant
Gibbs measure) and the one-sided Gibbs measure, that is, the marginal of the two-sided measure 
obtained by taking the restriction to $\NN$. Consider the transfer operator
$\L=\L_\phi$ from Theorem~\ref{thm:main}. Let $\nu\in\CM_\phi$ and let $\mu\in\CM(\X)$ be any
\emph{translation invariant} measure such that $\mu\vert_{\CF_n}\ll\nu\vert_{\CF_n}$, for all
$n\ge0$. For $x\in\X$ define the likelihood ratios $h_n(x)$, $n\ge0$, by
\begin{equation}\label{eq:defh}
  h_n(x)=\frac{\mu\left(\cyl xn\right)}{\nu\left(\cyl xn\right)},
\end{equation}
where $\cyl xn=\cyl x\Fn$. The limit of $h_n$ in~\eqref{eq:defh} is well-defined
$\nu$-almost everywhere by the martingale convergence theorem. If it exists in
$L^1(\nu)$ then $\mu \ll \nu$ and the limit $h$ is equal to the Radon-Nikodym derivative
$h=d\mu/d\nu$. As shown in the lemma below, we can then deduce the existence of an
eigenfunction $h$ in $L^1(\nu)$.

The following lemma states that if the sequence $h_n$ converges uniformly, then
the continuous limit function indeed is a strictly positive eigenfunction. The
continuity is of course an elementary consequence of uniform convergence by
Cauchy's theorem, and boundedness follows from the continuity of $h$ on the
compact set $X$. Uniform convergence also implies convergence in $L^1(\nu)$ and we
note that $\int h\; d\nu = 1$.
\begin{lemma}[Radon-Nikodym interpretation]\label{lem:RN}
  If $h_n(x)\to h(x)$ uniformly as $n\to\infty$ then $h$ is a continuous eigenfunction of
  $\L$ such that $\inf h(x) >0$.
\end{lemma}
\begin{proof}[Proof of Lemma~\ref{lem:RN}]
  That the Radon-Nikodym derivative $h=d\mu/d\nu$, if it exists, is necessarily an
  eigenfunction of the transfer operator $\L$ follows from
  \begin{eqnarray*}
    \int g \cdot h\d\nu &=& \int (g\circ \T) \cdot  h \d\nu \qquad(\mu=\mu\circ\T^{-1}) \\
               &=& \int \frac 1\lambda \L(g\circ\T\cdot h) \d\nu \qquad (\text{$\nu$ eigenmeasure})\\
               &=& \int g \cdot \left(\frac 1\lambda \L h \right) \d\nu,
  \end{eqnarray*}
  where the last equality follows from the definition of $\L$. This holds for
  all $g\in C(\X)$ if and only if $\L h = \lambda h$, $\lambda>0$, as elements of $L^1(\nu)$.

  We deduce that $\inf h>0$ by the following argument. For an $x$ such that
  $h(x)=0$ we have $\L h(x) = \sum_{a\in A} e^{\phi(ax)} h(ax) = 0$ and hence $h(ax)=0$
  for all $a\in\A$. Thus the set of zeros of $h$ is either empty or a dense subset
  of $\X$, implying that $\inf h>0$ by continuity and compactness.
\end{proof}

\subsection{The FK-Ising model}

We parameterise the \emph{Bernoulli graph model} $\eta(\g;p)\in\CM(\Gamma(V))$ by
edge probabilities $p:V\ii2\to[0,1]$, where $p(ij)=\P(\g_{ij}=1)$. Given $\J(k)$,
$k\in\NN$, as in Theorem~\ref{thm:main}, let $J_V(ij) := \J(|i-j|)$, $ij\in V\ii2$,
where $V\subset \ZZ$. We write $p=1-e^{-J_V}$ if
\begin{equation}\label{eq:pdef}
  p(ij) = 1 - e^{-J_V(ij)},\quad ij\in V\ii2.
\end{equation}

We obtain the \emph{FK-Ising model}
$\FK(x,\g; J_V) \in \CM(\X \times \Gamma(V))$ as a joint distribution of spin
configuration $x\in\X={\{-1,+1\}}^V$ and a random graph $\g\in\Gamma(V)$. The
pair $(x,\g)$ is compatible in the sense that no path in $\g$ connects vertices
of opposing spins. If $\alpha(x,\g)=\FK(x,\g;J_V)$ then $\alpha(x,\g)$ it is not
fully specified by a potential in the sense defined above. But, the marginal
$\alpha(x)$ of the spin sequence $x\in\X$ is an \emph{Ising model}
$\alpha(x)\in\CG(\Phi(x;J_V))$ with potential
\begin{equation}\label{eq:dysonfull}
  \Phi(x)=\Phi(x;J_V) = \sum_{ij\in V\ii2} J_V(ij) x_i x_j.
\end{equation}
As the marginal of $\g$, we obtain the \emph{random-cluster model} (FK-model)
\[ \alpha(\g) = \RC_{2}(\g;p=1-e^{-J_V})\] 
defined in \eqref{eq:rcdef} below. The conditional distribution
of $x$ given $\g$ is that of $x_i=x(C_\g(i))$, where
$(x(C):C\in \C(\g))\in{\{-1,+1\}}^{\C(\g)}$ has the uniform Bernoulli distribution
$\upsilon$.

With $r(x)=\sum_n r_n x_n$ as in condition~\eqref{eq:coshx} in
Theorem~\ref{thm:main}, we see that, conditioned on $\g\in\Gamma(\NN)$, the
distribution of $r(x)$ is that of a Rademacher series $r(x)=\sum_{C} x(C)\cdot r(C)$,
where $x(C)\in\{-1,+1\}$ are uniformly and independently sampled and $r(C)=\sum_{n\in C} r_n$. Thus
\(\Ex{e^{r(x)}\mid \g}=\prod_C \cosh(r(C))\) and we obtain
\begin{equation}\label{eq:coshxcosh}
	\int e^{r(x)} \d\nu(x) = \int \prod_{C\in\C(\g)} \cosh\left(r(C)\right) \d\nu(\g).
\end{equation}
For our purposes, the right hand side is more useful.

The following lemma expresses the cylinder probabilities of the Ising model in
terms of the random cluster model $\alpha(\g)$. For a graph $\g\in\Gamma(V)$, partial spin
$x\in {\{+1,-1\}}^S$, $S\subset V$, we say that $\g$ is \emph{compatible with $x$ at
  $F\Subset S$} if no path in $\g$ has endpoints $i,j\in F$ such that $x_i\not=x_j$.
Write $B_F(x,\g)\in\{0,1\}$ to indicate compatibility between $x$ and $\g$ at $F$
and $B_n(x,\g)$ if $F=\Fn$.
\begin{lemma}\label{lem:probcyl}
  For a FK-Ising distribution $\alpha(x,\g) = \FK((x,\g); J_V)$ and a fixed finite
  subset $F\Subset V$, we can express the probability of a cylinder $\cyl xF$ as
  \begin{equation}\label{eq:probcyl1}
    \alpha\left(\cyl xF\right) =  \int 2^{-\w_F(\g)} \, B_F(x,\g) \d\alpha(\g),
  \end{equation}
  where $\w_F(\g) := |\{ C\in\C(\g) : C\cap F \not=\emptyset \}|$ is the number of clusters
  in $\g$ that intersect $F$.
\end{lemma}
\begin{proof}
  Conditioned on $\g$ and the event that $\g$ is compatible with $x$ at $F$, the
  probability that the cluster-wise assignment of spins $\{x(C)\}$ gives rise to
  the cylinder $\cyl xF$ equals $2^{-\w_F(\g)}$.
\end{proof}

For general $q\ge1$, we obtain the random cluster model by modulating the
Bernoulli graph model $\eta(\g;p=1-e^{-J_V})$ with
$q^{\w(\g)}=e^{\w(\g)\log q}$, i.e.,
\begin{equation}\label{eq:rcdef}
  \RC_q(\g;p) = q^{\w(\g)} \ltimes \eta(\g;p) = \CG(\w(\g)\cdot\log q + \Log\eta(\g;p)).
\end{equation}
Although the potential $\w(\g)$ is discontinuous, the existence of a unique
element in $\RC_q(\g;p)$ is well established. Note also that
$\eta(\g;p) = \RC_1(\g;p)$.

From~\cite[Theorem 3.21, p.~43]{grimmett}, we obtain the following stochastic domination relations
\begin{align} \label{eq:rcstochdom}
    \RC_q(\g; p) &\prec \RC_{q'}(\g; p') \quad\text{when $p\le p'$ and $q\ge q'$,} \\
\label{eq:rcstochdom2}
    \RC_q(\g; p) &\prec \RC_{q'}(\g; p') \quad\text{when $\frac{p}{q(1-p)} \le \frac{p'}{q'(1-p')}$.}
\end{align}
It follows that we have
\begin{equation}\label{eq:twodom} 
  \eta(\g; \check{p}) \prec \RC_2(\g;p) \prec \eta(\g;p).
\end{equation}
The first domination relation in \eqref{eq:twodom} follows from \eqref{eq:rcstochdom2}, since
$\check{p} = p/(2-p)$ satisfies 
\[
\frac {\check p}{1\cdot(1-\check{p})} = \frac{p}{2\cdot(1-p)}. 
\]
The second domination relation in \eqref{eq:twodom} follows directly from \eqref{eq:rcstochdom}.

It follows from~\eqref{eq:rcstochdom} that there exists a critical
$\beta_c=\beta_c(J_V)\ge0$ such that for $\beta<\beta_c$ the probability of percolation
$\P(\g\in\percol)=0$ for the random graph $\RC_2(\g;p=1-e^{-\beta J_V})$ and
$\P(\g\in\percol)=1$ if $\beta>\beta_c$. This is the same critical $\beta$ for uniqueness of
the corresponding Ising model.

\subsection{Proof of Theorem~\ref{thm:main}}

In what follows, we write $\mu(x,\g)=\FK(x,\g; J_\ZZ)$ for the \emph{two-sided}
FK-Ising model and write $\nu(x,\g_+)=\FK(x,\g_+; J_\NN)$ for the \emph{one-sided}
model. For the marginals, we write $\mu(\bar x)$ and $\nu(x)$ for the corresponding
Ising models and denote the corresponding random cluster models by $\mu(\g)$ and
$\nu(\g_+)$. Our aim is to show the uniform convergence of the sequence $h_n(x)$,
defined by the translation invariant marginal distribution $\mu(x)$ of
$x = \bar x\vert_\NN$, and the one-sided Ising model $\nu(x)$ which is also the unique
eigenmeasure of $\L_\phi$. By Lemma~\ref{lem:RN} this implies the existence of a
continuous eigenfunction $h(x)$.

\subsubsection{The cut}
We consider a bipartition $V=V_-\uplus V_+$ of $V=\ZZ$, where $V_+=\NN$ and
$V_-=\ZZ\setminus\NN$. This cut leads to a unique decomposition of the graph $\g\in\Gamma(\ZZ)$
into three disjoint subgraphs
\begin{equation*}\label{graphdecomp}
  \g = \g_+\uplus\, \e\, \uplus\,\g_-.
\end{equation*}
Here $\g_\pm = \g\cap K(V_\pm)$ are the subgraphs induced on the parts and
$\e=\g\cap K(V_-,V_+)$ is the bipartite graph of edges $ij$ in $\g$ between
vertices $i\in V_-$ and $j\in V_+$. We also write $\ww$ for the union
$\ww:=\g_- \cup \e = \g\setminus \g_+$.

Consider the contracted graph
\[
  \tle^n = \e \bmod{(\Fn+\g\setminus\e)},
\]
where $\Fn+\g\setminus\e$ refers to the equivalence relation where $i\sim j$ if either
$i=j$, $\{i,j\}\subset \Fn$ or there is a path in the graph $\g\setminus\e=\g_+\uplus\g_-$
connecting $i$ and $j$. Then $\tle^n$ is a bipartite graph on vertex set
$\C(\g\setminus\e)=\C_-\uplus\C_+$, $\C_\pm := \C(\g_\pm)$ except that the $\w_n(\g_+)$ components
of $\g_+$ that intersect $\Fn$ join together into the vertex
\[
  \tl C_n =\tl C_n(\g_+) := \cup \{ C\in \C_+ : C\cap \Fn \not=\emptyset\}.
\]

Define for $n\geq 0$ the sequence
\begin{equation}\label{eq:Rdef}
    R_n(\g) = \cork \tle^n
\end{equation}
where $\tle^n$ is the contraction of $\e$ introduced above and the corank equals
the maximum number of edges that are removable without disconnecting clusters.
Since contraction increases the corank it is clear that the sequence $R_n(\g)$
increases. Let ${(a)}_+=\max\{a,0\}$. We can express the limit $R = \lim R_n$ as
\begin{equation}\label{eq:Reqsum}
  R(\g) = R(\ww) = \sum_{C\in \C_-} {\left(\deg(C,\e) - 1\right)}_+,
\end{equation}
since the limit graph of $\tle^n$ is a tree with root $\tl C_\infty = \NN$ and height
one. An edge $ij$ is then removable in the limit graph precisely when
$\deg(C_{\g_-}(i),\e)\ge2$ for the unique cluster $C_{\g_-}(i)\in\C_-$ that contains
$i\in V_-$. Let $\Firr=\Firr(\ww)$ denote the set of endpoints of paths of
$\ww=\g_- \cup \e$ connecting vertices in $V_+$. We have
\begin{equation}\label{eq:Rdelta}
  R_n(\g) = R(\ww),
\end{equation}
precisely when $\Firr \subset \Fn$.

Write $\eta(\e)=\eta(\e;p=1-e^{-J})$ and let $\check{\eta}(\e)$ be the Bernoulli graph model
\begin{equation}\label{tleta}
  \check{\eta}(\e) = 2^{-|\e|} \ltimes \eta(\e) = \eta(\e; \check{p}) \quad\text{where $\check{p}=p/(2-p)$}
\end{equation}
Let $\nu(\g_-)= \nu \circ \psi^{-1}=\RC(\g_-; p=1-e^{-\beta J_{V'}})$, $V'=\ZZ_-\setminus\NN$, refer to the
one-sided random cluster model $\nu(\g_+)$ under the mirror involution
$\psi: \ZZ\to\ZZ$, given by $j\mapsto-(j+1)$. Let also $\xi(\ww)$ denote the product
distribution
\[
  \xi(\ww) = \nu(\g_-) \otimes \check{\eta}(\e).
\]
\begin{lemma}\label{lem:cosh}
 Provided the condition~\eqref{eq:coshx} in Theorem~\ref{thm:main} holds 
 for the one-sided Ising model $\nu(x)$,  we have
  \begin{equation*}\label{eq:qqn}
	\int 2^{R(\ww)} \d\xi(\ww) < \infty.
  \end{equation*}
   In particular, we have $|\Firr|<\infty$, $\xi(\g)$-almost
  surely.
\end{lemma}
\begin{proof}[Proof of Lemma~\ref{lem:cosh}]
  Let $p(ij)=1-e^{-J(ij)}$ and $\check{p}=p/(2-p)$. For the Bernoulli distribution
  $\opn{Be}(\check{p}(ij))$ of $\e_{ij}\in\{0,1\}$, we have the following dominance
  relations
  \[
   \opn{Be}(\check{p}(ij)) \prec \opn{Be}(p(ij)) \prec \opn{Po}(J(ij))
  \]
  where $\opn{Po}(\lambda)$ refer to the Poisson distribution. For fixed $C\subset V_-$, it
  follows that the product distribution $\check{\eta}(\deg(C,\e)) \prec X(C)$, where
  $X(C) \sim \opn{Po(\lambda)}$, with $\lambda=r(\psi(C))$, where $\psi:\ZZ\to\ZZ$ is the mirror map
  $i\to-i-1$ mapping $V_-$ to $V_+=\NN$. Furthermore, if $X\sim\opn{Po}(\lambda)$ then
  \[
	\Ex{ 2^{{(X-1)}_+} } = e^{-\lambda} + \frac 12 \sum_{k=1}^\infty 2^k e^{-\lambda}\frac{\lambda^k}{k!} = \cosh(\lambda).
  \]
  Hence, if we condition of $\g_-$, we obtain that
  \[
	\Ex{2^{R(\ww)}\mid \g_- } \le \prod_{C\in\C_-} \cosh(r\circ\psi(C))
  \]
  and thus, with $\g = \psi\ii2 \g_-$,
  \begin{equation}\label{eq:bnd1}
  \int 2^{R(\ww)} \d\check{\eta}(\e) \d\nu(\g_-) \le \int \prod_{C\in\C(\g)} \cosh(r(C_n)) \d\nu(\g) < \infty
  \end{equation}
  with the finiteness due to the assumption~\eqref{eq:coshx} and the
  equality~\eqref{eq:coshxcosh}. Since $|\Firr|\le R(\ww)$ and we have shown that
  $2^R\in L^1(\xi)$, it follows that $\Firr$ is $\xi$-almost surely finite.
\end{proof}

\subsubsection{The factorised representation} We can use Lemma~\ref{lem:cosh} to
derive the following lemma that describes the factorisation of the two-sided
random cluster distributions implied by the graph
decomposition~\eqref{graphdecomp}.
\begin{lemma}\label{lem:factorisation}
  Under the same conditions as in Lemma~\ref{lem:cosh}, we have the following expression
  for the two-sided random cluster model
  \begin{equation}\label{eq:represent2}
	\mu(\g) = \frac1{K_0} \cdot 2^{R_0(\g)} \cdot \left(\nu(\g_+) \otimes \xi(\ww)\right),
  \end{equation}
  where $K_0=\int 2^{R_0(\g)} \d\nu(\g_+)\d\xi(\ww) < \infty$.
\end{lemma}
\begin{proof}
  If we consider the decomposition in~\eqref{graphdecomp}, it is clear that the
  Bernoulli distribution $\eta(\g)=\eta(\g; p)$ factorises into three disjoint Bernoulli graphs
  \begin{equation}\label{etafactor}
	\eta(\g) = \eta(\g_+) \otimes \eta(\e) \otimes \eta(\g_-).
  \end{equation}
  For finite graphs $H\subset G$, we have $\w(G)=\w(G\setminus H)-\opn{rank}(H\bmod{G\setminus H})$
  and, since this equality holds for potentials, we have
  \begin{equation}\label{eq:mat1}
	\w(\g) = \w(\g_+) + \w(\g_-) - |\e| + R_0(\g).
  \end{equation}
  By definition $\mu(\g) = 2^{\w(\g)} \ltimes \eta(\g)$ where the equality, as before, 
  states that $\mu(\g)$ is unique.

  Thus, \eqref{eq:modassoc} gives that
  \begin{align*}
	\mu(\g) & = 
            2^{\w(\g_+) + \w(\g_-) - |\e| + R_0(\g)} \ltimes
			\left(\eta(\g_+) \otimes \eta(\e) \otimes \eta(\g_-)\right)                                       \\
		  & = 2^{R_0} \ltimes
			\left(
            2^{\w(\g_+) + \w(\g_-) - |\e|} \ltimes
			\left(\eta(\g_+) \otimes \eta(\e) \otimes \eta(\g_-)\right)     \right)
   \end{align*}
  and \eqref{eq:moddistrib} gives that
  \begin{align*}
       \mu(\g)    & = 2^{R_0} \ltimes
			\left(
			(2^{\w(\g_+)} \ltimes \eta(\g_+))  \otimes (2^{-|\e|} \ltimes \eta(\e)) \otimes (2^{\w(\g_-)} \ltimes \eta(\g_-))
			\right)                                                                                           \\
		  & = 2^{R_0(\g)} \ltimes \left(\nu(\g_+) \otimes \xi(\ww)\right).
  \end{align*}
  The equality~\eqref{eq:represent2} follows from~\eqref{ccdef} and
  Lemma~\ref{lem:cosh}, since $2^{R_0}\le 2^R$ is in
  $L^1(\nu(\g_+)\otimes \xi(\ww))$.
\end{proof}

\subsubsection{The conclusion in the proof of Theorem~\ref{thm:main}}

Let, as in Lemma~\ref{lem:probcyl}, $B_n(x,\g) = B_\Fn(x,\g)$ indicate that the
spin $x$ and graph $\g$ are compatible at $F=\Fn$. Note that we can write the 
indicator $B_n(x,\g)$ as 
\begin{equation}\label{eq:Bnfac}
  B_n(x,\g) = A_n(x,\g) \cdot B_n(x,\g_+)
\end{equation}
where
\[
  A_n(x,\g) =
  \begin{cases}
    B_{\tl C_n}(x,\ww) & B_n(x,\g_+) = 1  \\
    1                & B_n(x,\g_+) = 0.
  \end{cases}
\]
Thus $A_n$ indicates that the compatibility is not killed by a path in
$\ww=\g_-\cup\e$. It also holds that $B_n(x,\g_+)=1$ and $\Firr\subset\Fn$
together implies that
\begin{equation}\label{eq:Bdelta}
  A_n(x,\g) = A(x,\ww) := B_{\Firr}(x,\ww),
\end{equation}
since any path in $\ww$ implying $A_n(x,\g)=0$ must have endpoints
$\{i,j\}\subset \Firr$ with $x_i\not= x_j$.

Let $\alpha_n(\g_+)\in\CM(\Gamma(\NN))$ be the distribution
\[
  \alpha_n(\g_+) = \frac{2^{-\w_n(\g_+)} \cdot B_n(x,\g_+) \cdot \nu(\g_+)} {\int 2^{-\w_n(\g_+)} \cdot B_n(x,\g_+) \cdot \d\nu(\g_+)}.
\]
From Lemma~\ref{lem:probcyl}, we deduce that
\begin{align*}
  h_n(x) &=
           \frac
           {\int 2^{-\w_n(\g)} B_n(x,\g) \d\mu(\g)}
           {\int 2^{-\w_n(\g_+)} B_n(x,\g_+) \d\nu(\g_+)}
  \\[4pt]
         &=
           \frac
           {\frac1{K_0}\int 2^{\w_n(\g_+)-\w_n(\g)} \cdot 2^{R_0(\g)}\cdot B_n(x,\g)\cdot  2^{-\w_n(\g_+)}\d\nu(\g_+)\d\xi(\ww)}
           {\int 2^{-\ w_n(\g_+)} B_n(x,\g_+) \d\nu(\g_+)}
           \qquad
           \text{by~\eqref{eq:represent2}}
  \\[4pt]
         & = \frac 1{K_0}\,\int 2^{\w_n(\g_+)-\w_n(\g)}
           \cdot 2^{R_0(\g)} A_n(x,\g)\d\alpha_n(\g_+)\d\xi(\ww)
           \qquad\text{by~\eqref{eq:Bnfac}}.
\end{align*}

Note that
\begin{equation}\label{eq:wR}
  \w_n(\g_+) - \w_n(\g) = R_n(\g) - R_0(\g)
\end{equation}
since both sides equal the rank of the subgraph of $\tle^0$ consisting of edges
$ij\in H$ with one endpoint in $\tl C_n$. From~\eqref{eq:wR}, we deduce that
\begin{equation}\label{eq:hnint}
	h_n(x) = \frac 1{K_0} \cdot \int A_n(x,\g) \cdot 2^{R_n(\g)} \cdot \d\alpha_n(\g_+)\d\xi(\ww).
\end{equation}

Let $N=N(\ww)$ be the minimum $n$ such that $\Firr\subset\Fn$ and note that
Lemma~\ref{lem:cosh} implies that $\xi(N\ge n)\to 0$ as $n\to\infty$. On account
of~\eqref{eq:Rdelta} and~\eqref{eq:Bdelta}, it follows that the integrand
\(A_n(x,\ww) \cdot 2^{R_n(\g)}\) in~\eqref{eq:hnint} equals $A(x,\ww)\cdot 2^{R(\ww)}$
on the event $N<n$, since $B_n(x,\g_+)=1$, $\alpha_n$-almost surely. Let
\[
  h(x) := \frac 1{K_0}\int A(x,\ww) 2^{R(\ww)} \d\alpha_n(\g_+)\d\xi(\ww)
  = \frac1{K_0}\,\int A(x,\ww) 2^{R(\ww)} \d\xi(\ww).
\]
Thus
\begin{align*}
  |h_n(x) - h(x)|
  & \le \frac 1{K_0} \,
	\int \left|A_n(x,\ww)\cdot2^{R_n(\g)}-A(x,\ww)\cdot 2^{R(\g)}\right|\d\alpha_n(\g_+)\d\xi(\ww)   \\
  & = \frac 1{K_0}\,
	\int_{N>n} \left|A_n(x,\ww)\cdot2^{R_n(\g)}-A(x,\ww)\cdot2^{R(\g)}\right|\d\alpha_n(\g_+)\d\xi(\ww).
\end{align*}
Since
\[ |A_n(x,\ww) \cdot 2^{R_n(\g)} - A(x,\ww)\cdot 2^{R(\g)}|\le 2^{R(\ww)}\in L^1(\xi), \]
we conclude that
\begin{align*}
  |h_n(x) - h(x)| & \le  \frac1{K_0} \int_{N>n} 2^{R(\ww)} \d\xi(\ww)
				  \le \frac1{K_0}\cdot \xi(N>n) \cdot \int 2^{R}\d\xi,
\end{align*}
where, on account of Lemma~\ref{lem:cosh}, the right hand side tends to zero
with a rate independent of $x$. {\qed}

\subsection{Proof that Theorem~\ref{thm:main} implies Theorem~\ref{main}}

Assume $\g=\g_+\in\Gamma(\NN)$ with distribution $\nu(\g)= \RC(\g;p=1-e^{-\beta J_\NN})$.
Recall that $J(ij)=\J(|i-j|)$ satisfies the square summability condition
\(\sum r_n^2 < \infty\) with $r_n=\sum_{k=n+1} \J(k)$. Order the elements of $\C=\C(\g)$ as
$\C=\{C_0,C_1,\dots\}$ so that \(0 = \iota_0 < \iota_1 < \dots\) where
\[
  \iota_n =\min(i\in C_n)= \inf \{ i : i \not\in C_1\cup \dots \cup C_{n-1}\}.
\]

We first show that if the cluster size distribution of $C_0$ has
\emph{exponentially decreasing tails}, i.e.\ if, for some $K>0$ and some $c>0$,
we have
\begin{equation}\label{eq:geometric}
  \P(|C_0| > n)  \le Ke^{-cn},
\end{equation}
then this implies the condition~\eqref{eq:coshx} or equivalently,
by~\eqref{eq:coshxcosh}, that
\begin{equation}\label{eq:cosh}
  \Ex{\prod_{C\in\C(\g)} \cosh(r(C))}<\infty.
\end{equation}
Thus~\eqref{eq:geometric} implies the conditions of Theorem~\ref{thm:main} hold
and thus the existence of a continuous eigenfunction.

If we condition the random cluster model on the clusters $\{C_0,\dots,C_{k-1}\}$
that partition $\halfopen{0}{\iota_k}$ then the distribution of the remaining graph
$\g\left[\overline{\left(\cup_{j=1}^{k-1} C_j\right)}\right]$ is the random cluster
model with edge probabilities $p'(ij) = p(ij)\ett_{i,j\not\in \cup \C_n}$. (See
e.g.~\cite{berghaggkahn}.) It follows that the conditional distribution of $C_k$,
given $C_1,\dots, C_{k-1}$, is stochastically dominated by the distribution of $C_0$ shifted $\iota_k$ steps to the right. In particular, it follows that the
conditional distribution of $|C_k|$ has exponentially decreasing tails. Thus,
for some $K>0$ and $c>0$ as in~\eqref{eq:geometric}
\begin{align}
  \Ex{|C_k|^n\mid C_1,\dots,C_{k-1}}
  &\le \int_0^\infty \P(|C_0|^n \ge x) \d x \nonumber\\
  & \le K \cdot \int_0^\infty e^{-c x^{1/n}} \d x
  = K \cdot \frac {n!}{c^n}.\label{eq:moment}
\end{align}

Since $r_n$ is a decreasing sequence, we have ${r(C_k)} \le r_{\iota_k}\cdot|C_k|$ and the
Cauchy-Schwarz inequality implies that ${r(C_k)}^2\le R\cdot |C_k|$ where
$R=\sum_{i=0}^\infty r_i^2$. Thus,
\begin{align*}
  {\cosh(r(C_k))} &=1 + \sum_{n=1}^\infty \frac{{r(C_k)}^{2n}}{(2n)!}
  \le 1 + {r^2_{\iota_k}\cdot|C_k|^2} \sum_{n=1}^\infty \frac{{r(C_k)}^{2n-2}}{(2n)!} \\
  &\le 1 + {r^2_{\iota_n}} \cdot \sum_{n=1}^\infty \frac{R^{n-1} |C_k|^{n+1}}{(2n)!}.
\end{align*}
Taking the conditional expectation, using~\eqref{eq:moment}, gives
\begin{align*}
  \Ex{\cosh(r(C_k))|C_1,\dots,C_{k-1}}
  &\le 1 + {r^2_{\iota_k}} \cdot \sum_{n=1}^\infty \frac{K R^{n-1} \cdot (n+1)! \cdot c^{-{(n+1)}}}{(2n)!} \\
  &\le 1 + {r^2_{\iota_k}} M
\end{align*}
where a term-wise comparison gives $M<\infty$. (E.g.\ using that
$2^n\cdot{(n!)}^2\le{(2n)!}$.)

We obtain
\begin{align*}
  \Ex{ \prod_{n=0}^\infty \cosh(r(C_n)) }
  &= \Ex{  \prod_{n=0}^\infty \Ex{\cosh(r(C_n))\mid C_1,\dots, C_{n-1} }} \\
  &\le e^{ M \cdot \sum_{n=1}^\infty r^2_{\iota_n}} < e^{M R} < \infty,
\end{align*}
and thus we have shown that $\eqref{eq:geometric}\implies\eqref{eq:cosh}$.

Finally, we need to show that the condition $\beta<\beta_c(J_\ZZ)$ of Theorem~\ref{main}
implies~\eqref{eq:geometric}. 

Since the weighting $J_\ZZ$ is vertex-transitive, the result in Hutchcroft~\cite{Hutch} says that if $\beta<\beta_c(J_\ZZ)$, then for
the two-sided model $\mu(\g)$, the distribution $\mu(|C_\g(\orig)|)$ of the size of the cluster containing any $\orig\in\ZZ$ has exponentially decreasing tails. Since
$\nu(\g_+)\prec \mu(\g[\NN])$ and $C_0=C_{\g_+}(0) \subset C_{\g}(0) \cap \NN$, the
condition~\eqref{eq:geometric} readily follows for the one-sided model $\nu$. {\qed}

\noindent
\textbf{Acknowledgements}. We would like to thank Noam Berger, Evgeny
Verbitskiy, and Aernout van Enter for valuable comments. The second author wishes to thank the Knut and
Alice Wallenberg Foundation for financial support. The third author acknowledges
the ERC Grant 833802--Resonances.


\begin{thebibliography}{999}

	\bibitem{ACCN} M.\ Aizenman, J.\ Chayes, L.\ Chayes and C.\ Newman,
	Discontinuity of the magnetization in the one-dimensional $1/|x-y|^2$ Ising
	and Potts models, {\em J.\ Statist.\ Phys.\ } {\textbf{50}} (1988), 1--40.

	\bibitem{AN} M.\ Aizenman and C.\ Newman, Tree Graph Inequalities and Critical
	Behavior in Percolation Models, {\em J.\ Statist.\ Phys.\ } {\textbf{36}} (1984),
	107--143.

	\bibitem{aoun} Y.\ Aoun, Sharp asymptotics of correlation functions in the subcritical
	long-range random cluster and Potts models, \emph{Electron.\ Commun.\ Probab.\ } {\textbf 26} (2021),
	No.\ 22, 9pp.

	\bibitem{berbee89} H. Berbee, Uniqueness of Gibbs measures and absorption
	probabilities, {\em Ann.\ Probab.\ } {\textbf{17}} (1989), no.\ 4, 1416--1431.

	\bibitem{berghaggkahn} J.\ van den Berg, O.\ H\"aggstrom, and J.\ Kahn,
	Some conditional correlation inequalities for percolation and related
	processes {\em Random Structures and Algorithms}, {\textbf{29}} (4)
	(2006), 417--435.

	\bibitem{berger} N.\ Berger, C.\ Hoffman and V.\ Sidoravicius, Nonuniqueness
	for specifications in $l^{2+\epsilon}$, {\em Ergodic Theory Dynam.\ Systems} {\textbf{38}}
	(2018), no.\ 4, 1342--1352.

	\bibitem{berger2} N.\ Berger, D.\ Conache, A.\ Johansson, and A.\ \"Oberg,
	Doeblin measures -- uniqueness and mixing properties, \emph{Probability Theory Related Fields}, https://link.springer.com/article/10.1007/s00440-024-01356-3.

	\bibitem{vanenter} R.\ Bissacot, E.O.\ Endo, A.\ C.\ D.\ van Enter, and A.\ Le
	Ny, Entropic Repulsion and Lack of the $g$-measure Property for Dyson Models,
	{\em Comm.\ Math.\ Phys.\ } {\textbf{363}} (2018), 767--788.

	\bibitem{doeblin} W.\ Doeblin and R.\ Fortet, Sur des ch{\^ai}nes {\`a}
	liaisons compl{\`e}tes, {\em Bull.\ Soc.\ Math.\ France} {\textbf{65}}
	(1937), 132--148.

	\bibitem{duminil2} H.\ Duminil-Copin, A.\ Raoufi, and V.\ Tassion, Sharp phase
	transition for the random-cluster and Potts models via decision trees,
	\emph{Ann.\ of Math.\ (2)} {\textbf{189}}
	(2019), no.\ 1, 75--99.
	
	\bibitem{dyson} F.\ J.\ Dyson, Existence of a phase-transition in a one-dimensional Ising ferromagnet,
	\emph{Comm.\ Math.\ Phys.\ } \textbf{12} (1969), no.\ 2, 91--107.

	\bibitem{vanny0} E.O.\ Endo, A.\ C.\ D.\ van Enter, and A.\ Le Ny, The roles
	of random boundary conditions in spin systems, in \emph{In and Out of
	  Equilibrium 3: Celebrating Vladas Sidoravicius}, Springer (2021), 371--381.

	\bibitem{vanny} E.O.\ Endo, A.\ C.\ D.\ van Enter, and A.\ Le Ny, On long
	range Ising models with random boundary conditions, arXiv:2405.08374.

	\bibitem{verb-van-enter2} A.C.D.\ van Enter, R.\  Fern\'andez, M.\ Makhmudov,
	and E.A.\ Verbitskiy, On an extension of a theorem by Ruelle to long-range potentials,
	arXiv:2404.07326.

	\bibitem{verb-van-enter} A.C.D.\ van Enter and E.A.\ Verbitskiy, On the
	Variational Principle for Generalized Gibbs Measures, {\em Markov Processes
			and Related Fields} {\textbf{10}} (2004), no.\ 3, 411-434.
	
	 \bibitem{tak} C.\ Gallesco, S.\ Gallo, and D.\ Y.\ Takahashi, Dynamic uniqueness
  for stochastic chains with unbounded memory, {\em Stochastic Process.\ Appl.\ } {\textbf {128}} (2018), 689--706.

	\bibitem{grimmett} G.\ R.\ Grimmett, {\em The Random Cluster Model}, Springer
	2006.

	\bibitem{Hutch} T.\ Hutchcroft, New critical exponent inequalities for percolation
	and the random cluster model, \emph{Probab.\ Math.\ Phys.\ } {\textbf 1} (2020), 147--165.

	\bibitem{johob} A.\ Johansson and A. \"Oberg, Square summability of variations
	of $g$-functions and uniqueness of $g$-measures, {\em Math.\ Res.\ Lett.\ }
		{\textbf{10}} (2003), no.\ 5--6, 587--601.

	\bibitem{johob2} A.\ Johansson, A.\ \"Oberg and M.\ Pollicott, Countable state
	shifts and uniqueness of $g$-measures, {\em Amer.\ J.\ Math.\ } {\textbf{129}}
	(2007), no.\ 6, 1501--1511.

	\bibitem{johob4} A.\ Johansson, A.\ \"Oberg and M.\ Pollicott, Phase
	transitions in long-range Ising models and an optimal condition for factors of
	$g$-measures, {\em Ergodic Theory Dynam.\ Systems} {\textbf{39}} (2019), no.\ 5,
	1317--1330.

	\bibitem{keane} M.\ Keane, Strongly mixing $g$-measures, {\em Invent.\ Math.\
		} {\textbf{16}} (1972), 309--324.

	\bibitem{pan} C.\ Panagiotis, {\em Interface theory and Percolation}, PhD
	Thesis, University of Warwick 2020.
	
	\bibitem{quas} A.\ N.\ Quas, Rigidity of continuous coboundaries, \emph{Bull.\ London Math.\ Soc.\ } \textbf{29} (1997), no.\ 5, 595--600.

	\bibitem{ruelle}
	D.\ Ruelle, Statistical mechanics of a one-dimensional lattice gas, \emph{Comm.\
		Math.\ Phys.\ }, {\textbf{9}} (1968), 267--278.

	\bibitem{sin} Ya.\ G.\ Sinai, Gibbs measures in ergodic theory, {\em Russian
			Mathematical Surveys} {\textbf{27}} (4) (1972), 21--69.

	\bibitem{walters1} P.\ Walters, Ruelle's operator theorem and $g$-measures,
	{\em Trans.\ Amer.\ Math.\ Soc.} {\textbf{214}} (1975), 375--387.
	
	 \bibitem{walters2} P.\ Walters, Convergence of the Ruelle operator for a function satisfying Bowen's condition, {\em
    Trans.\ Amer.\ Math.\ Soc.\ }  {\textbf {353}} (2000), no.\ 1, 327--347.

\end{thebibliography}
\end{document}